\documentclass[preprint,11pt]{elsarticle}
\usepackage[T1]{fontenc}
\usepackage{lmodern,microtype}
\usepackage{amsmath,amssymb,amsthm,mathtools}
\usepackage{booktabs,array}
\usepackage{geometry}
\usepackage[hidelinks]{hyperref}
\hypersetup{pdftitle={A2-frames, negacyclic glue, and extremal strongly 6-modular lattices},pdfauthor={Ian Teixeira}}

\newtheorem{theorem}{Theorem}[section]
\newtheorem{lemma}[theorem]{Lemma}
\newtheorem{proposition}[theorem]{Proposition}

\theoremstyle{remark}
\numberwithin{equation}{section}
\allowdisplaybreaks[2]

\newcommand{\Z}{\mathbb Z}
\newcommand{\F}{\mathbb F}
\newcommand{\ct}{\operatorname{ct}}
\newcommand{\wt}{\operatorname{wt}}
\newcommand{\norm}[1]{\lVert#1\rVert}
\newcommand{\typea}{6\mathrm a}
\newcommand{\eps}{\epsilon}
\newcommand{\GammaCode}{\Gamma}

\journal{Journal of Number Theory}
\biboptions{sort&compress}
\setcitestyle{numbers,square}

\begin{document}
\begin{frontmatter}
\title{$A_2$-frames, negacyclic glue, and extremal strongly $6$-modular lattices}
\author[ucsd]{Ian Teixeira\corref{cor1}}
\ead{iteixeira@ucsd.edu}
\cortext[cor1]{Corresponding author.}
\address[ucsd]{Department of Mathematics, University of California San Diego,
La Jolla, CA 92093, USA}

\begin{abstract}
The classical $A_2$-frame construction of the Coxeter--Todd lattice uses
the hexacode as its glue.  Here we replace the inert Eisenstein reduction
at $2$ by the ramified Gaussian chain ring $S=\Z[i]/(1+i)^3$.  For odd
$\ell$, a one-generator cyclic graph code over $S$ lifts from an
$A_2(4)^{2\ell}$-frame to an even strongly $6$-modular lattice $L_\ell$ of
rank $4\ell$.  In the real frame ordering the same glue is a one-generator
negacyclic code.  The lattices $L_5$ and $L_7$ have minima $6$ and $8$, respectively, resolving the open type-\(6a\) existence cases in dimensions \(20\) and \(28\) in the Nebe–Sloane catalogue.  After completing the square, the
minimum calculation becomes a four-symbol carry problem; for $L_7$ the key
relation is the boundary map of a $7$-cycle.  For $\ell=3$, $L_3$ has an
index-$8$ overlattice isometric to the Coxeter--Todd lattice.
\end{abstract}

\end{frontmatter}

\section{Introduction}
The Coxeter--Todd lattice $K_{12}$ gives the basic model for the
construction in this paper \cite{ConwaySloane83,ConwaySloane99}.  Identifying
$A_2$ with the Eisenstein
integers $\Z[\omega]$, the prime $2$ is inert and
$\Z[\omega]/2\Z[\omega]\cong\F_4$.  Relative to an $A_2(2)^6$-frame, the
$2$-primary glue of $K_{12}$ is the Hermitian self-dual $[6,3,4]_4$
hexacode \cite{ConwaySloane99,SelfDual}.  In this form the lattice question
becomes a finite coding question: the code determines the index, and its
distance rules out roots.

Here the Eisenstein reduction at $2$ is replaced by a ramified Gaussian
one.  Put $\pi=1+i$.  Since $2=-i\pi^2$, we work over the chain ring
\[
 S=\Z[i]/\pi^3\Z[i],\qquad S/\pi S\cong\F_2,
\]
which has eight elements and additive group $\Z/4\oplus\Z/2$.  For odd
$\ell$ we construct a one-generator cyclic graph code
$\GammaCode_\ell\subset S^{2\ell}$ whose lift from an
$A_2(4)^{2\ell}$-frame is a strongly $6$-modular lattice $L_\ell$.  Although
the coefficient ring is Gaussian, the real frame is still an $A_2$-frame:
after the basis change in Section~\ref{sec:construction}, the real and
imaginary coefficient directions give the $2\ell$ orthogonal $A_2(4)$
blocks.  In the real $A_2$-block ordering, the same glue is a one-generator
negacyclic module over $(\Z/4)[t]/(t^{2\ell}+1)$.  For related modular-lattice
constructions from codes over finite rings, see Bachoc \cite{Bachoc97}.

Write $A_2(s)$ for the root lattice $A_2$ with its bilinear form multiplied
by $s$.  An $A_2(s)^m$-\emph{frame} in a rank-$2m$ lattice is a full-rank
orthogonal direct sum of $m$ copies of $A_2(s)$.  We call such a frame
\emph{transitive} when its stabilizer in the lattice automorphism group
acts transitively on the $m$ blocks.

For an even positive-definite lattice $L$, put
$L^{\#,d}=L^\#\cap d^{-1}L$ for $d\mid6$.  The lattice is
\emph{strongly $6$-modular} if
$L\cong\sqrt d\,L^{\#,d}$ for every $d\mid6$.
The modular-form bounds of Quebbemann and Rains--Sloane
\cite{Quebbemann95,Quebbemann97,RainsSloane98} give
\begin{equation}\label{eq:bound}
 \min L\le 2+2\lfloor \dim L/8\rfloor.
\end{equation}
A lattice attaining this bound is called \emph{extremal}.

\begin{theorem}\label{thm:main}
The lattices $L_5$ and $L_7$ are even strongly $6$-modular of type
$\typea$, with
\[
 \det L_5=6^{10},\quad \min L_5=6,
 \qquad
 \det L_7=6^{14},\quad \min L_7=8.
\]
Moreover $L_5$ contains a transitive $A_2(4)^{10}$-frame and $L_7$
contains a transitive $A_2(4)^{14}$-frame.  In particular, both lattices
are extremal.
\end{theorem}

The type-$\typea$ entries in dimensions $20$ and $28$ are listed as open
in the maintained Nebe--Sloane catalogue \cite{Catalogue}.  The same family has extremal members in the previously
known dimensions $4$ and $12$:
\[
\begin{array}{c|rrrr}
 \ell&1&3&5&7\\\hline
 \dim L_\ell&4&12&20&28\\
 \min L_\ell&2&4&6&8.
\end{array}
\]
Only the last two existence statements are new.

Nebe's classification in dimension $16$ uses a $p$-maximal reduction.  If
$L$ is extremal strongly $6$-modular of rank $16$, then one may pass to an
even $3$-modular lattice $M$ with
\[
 L\subset M=M^{\#,2}\subset L^{\#,2},
 \qquad \min M\ge4.
\]
In rank $16$ such an $M$ is itself extremal, and there are only six possible
$3$-modular parents; the eight extremal strongly $6$-modular lattices are
then obtained by enumerating the appropriate $2$-power-index sublattices
\cite{Nebe21}.  This is still a small search in rank $16$, but not in ranks
$20$ and $28$.  In rank $20$ the parent is again extremal $3$-modular, but
the maintained catalogue already records at least $100$ such lattices
rather than a complete classification \cite{Catalogue}.  In rank $28$ the
reduction only forces $\min M\ge4$, below the extremal minimum $6$ for a
$3$-modular lattice, so the parent need not even be extremal.  The required
indices also grow from $2^4$ in rank $16$ to $2^5$ and $2^7$ in ranks $20$
and $28$.  Our construction avoids this parent enumeration: we prescribe a
transitive $A_2$-frame and its ramified one-generator glue from the start.

Section~\ref{sec:code} defines the ramified cyclic glue, and
Section~\ref{sec:construction} gives the Hermitian realization, the
transitive frame, and the partial dualities.  Sections~\ref{sec:metric}--
\ref{sec:minima} prove the minimum bounds.  Section~\ref{sec:k12bridge}
identifies the rank-$12$ relation with the Coxeter--Todd lattice.

\section{The ramified cyclic glue}\label{sec:code}
Let
\[
 R=\Z[i][x]/(x^\ell-1),\qquad
 J=1+x+\cdots+x^{\ell-1},
\]
with involution $\bar i=-i$, $\bar x=x^{-1}$.

We use the directed-cycle perturbation of the uniform Gaussian-unit
pattern.  Number the cycle so that its permutation is multiplication by
$x$ and set
\begin{equation}\label{eq:kdef}
 k=J-1-(1-i)x-(1+i)x^{-1}.
\end{equation}
We use the same formula for every odd $\ell$.

Let $S_\ell=S[x]/(x^\ell-1)$.  Define the graph code
\begin{equation}\label{eq:graphcode}
 \GammaCode_\ell=S_\ell(k+\pi,1)\subset S_\ell^2.
\end{equation}
After expanding coefficients, $\GammaCode_\ell$ is a rank-$\ell$
systematic code over $S$ with generator matrix
\[
 \bigl[\operatorname{Circ}(k+\pi)\mid I_\ell\bigr].
\]
Thus the circulant is the left-hand block: its diagonal entry is $\pi$,
its two directed cycle entries are $i$ and $-i$, and every other
off-diagonal entry is $1$.

This is not obtained from the hexacode construction by extending scalars.
The residue field of $S$ is $\F_2$, and $\GammaCode_\ell$ has mixed
additive exponent.  Strong $6$-modularity comes instead from explicit
partial similarities of the lifted lattice, not from Hermitian
self-duality of $\GammaCode_\ell$ for the ordinary form on $S^{2\ell}$.

\section{Hermitian realization, frames, and partial dualities}\label{sec:construction}
Express $a\in R$ in the coefficient basis
$1,x^{-1},\ldots,x^{1-\ell}$ as
$a=\sum_j a_jx^{-j}$ and set
\[
 \norm a_0^2=\ct(a\bar a)=\sum_j|a_j|^2.
\]
In addition to $k$ in \eqref{eq:kdef}, define
\begin{equation}\label{eq:kbc}
 b=(3+i)+2(1-i)k,
 \qquad
 c=2+k+k^2.
\end{equation}
Then $\bar k=k$ and
\begin{equation}\label{eq:detidentity}
 b\bar b+6=8c,
\end{equation}
since $b\bar b=10+8k+8k^2$.

On $R^2$ consider the coefficient lattice $L_\ell$ with Hermitian matrix
\begin{equation}\label{eq:H}
 H=\begin{pmatrix}8&b\\ \bar b&c\end{pmatrix}.
\end{equation}
For $z,w\in R^2$ define the underlying real bilinear form and quadratic
form by
\[
 \langle z,w\rangle
 =\operatorname{Re}\ct\!\left(\bar z^{\,t}Hw\right),
 \qquad Q(z)=\langle z,z\rangle.
\]
In coordinates $z=(u,v)^t$ this is
\[
 Q(u,v)=\ct\!\left((\bar u,\bar v)H\binom uv\right),
\]
since the latter is real.  The determinant identity gives the completed square
\begin{equation}\label{eq:square}
 8Q(u,v)=\norm{8u+bv}_0^2+6\norm v_0^2.
\end{equation}
Hence $L_\ell$ is positive definite of real rank $4\ell$.  Since $H$ has
Gaussian-integral coefficients, the real bilinear form $\langle\ ,\ \rangle$ is integral.  For $\ell\ge3$, the constant coefficient of $k$ is
zero and its other $\ell-1$ coefficients are Gaussian units, so
$\ct(c)=\ell+1$.  For $\ell=1$ the overlapping terms give $k=-2$ and
$c=4$.  In either case the real Gram matrix has even diagonal.

We use the following elementary local criterion.
\begin{lemma}\label{lem:duality}
Let $L$ be even of rank $n$ and determinant $6^{n/2}$. Suppose that, for
$p=2,3$, a similarity $\Phi_p:L\to L$ of multiplier $p$ satisfies
$(\Phi_pL,L)\subseteq p\Z$. Then $L$ is strongly $6$-modular.
\end{lemma}
\begin{proof}
Put $L_{(p)}=L\otimes\Z_p$.  The pairing condition gives
$\Phi_pL_{(p)}\subseteq pL_{(p)}^{\#}$. Both sides have determinant
valuation $3n/2$, so equality holds. At every other prime $\Phi_p$ is an
automorphism. Thus $\Phi_pL=pL^{\#,p}$. At $2$, the map $\Phi_3$ preserves
$L_{(2)}$; at $3$, the map $\Phi_2$ preserves both $L_{(3)}$ and
$L_{(3)}^\#$, since its multiplier is a unit. Composing these local
identities gives $\Phi_2\Phi_3L=6L^\#$.
\end{proof}

For an even strongly $6$-modular lattice, the $3$-primary discriminant
space $L_{(3)}^\#/L_{(3)}$ carries the nondegenerate ternary form
\[
 (u+L_{(3)},v+L_{(3)})\longmapsto 3(u,v)\pmod3.
\]
In ranks congruent to four modulo eight, type $\typea$ denotes the square
determinant class of this form \cite{Catalogue,Juergens15}.

\begin{proposition}\label{prop:construction}
For every positive odd $\ell$, the lattice $L_\ell$ is even and strongly
$6$-modular, with
\[
 \det L_\ell=6^{2\ell}.
\]
It contains an $A_2(4)^{2\ell}$-frame $\mathcal F_\ell$ of index
$2^{3\ell}$, and its ternary discriminant form has square determinant
class. In the catalogue convention for ranks $4\pmod8$, this is type
$\typea$.
\end{proposition}
\begin{proof}
Put
\begin{equation}\label{eq:frame}
 F=\begin{pmatrix}1&-(1+i)-k\\0&2(1+i)\end{pmatrix}.
\end{equation}
A direct multiplication using \eqref{eq:detidentity} gives
\[
 F^*HF=\begin{pmatrix}8&-4\\-4&8\end{pmatrix}.
\]
The columns and their Gaussian coefficient shifts span
$\mathcal F_\ell=FR^2$, consisting of $2\ell$ mutually orthogonal copies
of $A_2(4)$. Since $N(2(1+i))=8$, the frame has index
$8^\ell=2^{3\ell}$, and therefore
\[
 \det L_\ell=\frac{48^{2\ell}}{2^{6\ell}}=6^{2\ell}.
\]

Take $\Phi_2(z)=\pi z$.  Every entry of $H$ is divisible by $\pi$.
Modulo $\pi$ one has $i=1$ and $k=J-1$; since $\ell$ is odd, $J^2=J$, so
$c=0$ modulo $\pi$.  Thus $\bar\pi H$ is divisible by $2$, giving
$(\Phi_2L_\ell,L_\ell)\subseteq2\Z$.

Put
\[
 \Omega=\begin{pmatrix}0&1\\-1&0\end{pmatrix}.
\]
The $2\times2$ adjugate identity is
\begin{equation}\label{eq:adjugate}
 \bar H\Omega H=6\Omega.
\end{equation}
Put
\[
 T=\bar\pi^{-1}\Omega\bar H,
 \qquad \Phi_3(z)=T\bar z.
\]
This is integral because $\bar H$ is divisible by $\bar\pi$.  The
adjugate identity gives
\[
 \bar T^{\,t}HT=3\bar H,
 \qquad
 \bar T^{\,t}H=-3\bar\pi\Omega.
\]
Hence the semilinear map $\Phi_3$ has multiplier $3$ and
$(\Phi_3L_\ell,L_\ell)\subseteq3\Z$.  Lemma~\ref{lem:duality} proves
strong $6$-modularity.

Finally, the extension from $\mathcal F_\ell$ has $2$-power index, so the
$3$-adic completion is unchanged.  On one $A_2(4)$ block choose a root
basis $e,f$ with Gram matrix
\[
 \begin{pmatrix}8&-4\\-4&8\end{pmatrix}.
\]
Its ternary discriminant group is generated by
$y=(2e+f)/3$, and
\[
 3(y,y)=8\equiv2\pmod3.
\]
Thus each block contributes determinant class $2$, and on $2\ell$ blocks
the determinant class is $2^{2\ell}$, a square.
\end{proof}

For $z\in L_\ell$, we use \emph{quarter-frame coordinates} for the
class of
\[
 4F^{-1}z \pmod{4R^2}.
\]
This depends only on $z+\mathcal F_\ell$, since replacing $z$ by
$z+Fr$ changes $4F^{-1}z$ by $4r$.  Thus $L_\ell/\mathcal F_\ell$ is
represented as a code in $(R/4R)^2$.

\begin{proposition}[Frame glue]\label{prop:glue}
For $\ell\ge3$, multiplication by $\bar\pi$ identifies the graph code
$\GammaCode_\ell$ of \eqref{eq:graphcode} with the full frame glue
$\mathcal C_\ell=L_\ell/\mathcal F_\ell$ in quarter-frame coordinates:
\begin{equation}\label{eq:glue}
 \mathcal C_\ell
 =\bar\pi\,\GammaCode_\ell
 =R_4\bigl(2+(1-i)k,1-i\bigr)\pmod4,
\end{equation}
where $R_4=(\Z[i]/4\Z[i])[x]/(x^\ell-1)$.  Hence
\[
 \mathcal C_\ell\cong(\Z/4)^\ell\oplus(\Z/2)^\ell.
\]
Multiplication by $ix$ preserves the frame and induces a signed
$2\ell$-cycle on its $A_2$-blocks.  In that block order the glue is a
one-generator module over $(\Z/4)[t]/(t^{2\ell}+1)$.
\end{proposition}
\begin{proof}
The inverse of \eqref{eq:frame} gives
\[
 F^{-1}\binom uv=
 \binom{u+(\pi+k)v/(2\pi)}{v/(2\pi)}.
\]
Multiplying by four and reducing modulo four gives
\[
 \mathcal C_\ell=R_4\bigl(\bar\pi(k+\pi),\bar\pi\bigr).
\]
Since $4$ is a unit times $\pi^4$, multiplication by $\bar\pi$ identifies
$S=\Z[i]/\pi^3$ with the ideal $\bar\pi\Z[i]/4\Z[i]$.  This proves the
first assertion and the additive structure.  Multiplication by $ix$ is an
isometry of $L_\ell$: it is scalar multiplication by the Gaussian unit
$i$ followed by the cyclic coefficient shift $x$, and both preserve
$\langle\ ,\ \rangle$.  On the frame, multiplication by $i$ interchanges
the real and imaginary $A_2$-blocks with one sign change, while
multiplication by $x$ cycles the Gaussian positions.  For odd
$\ell$, $ix$ has one orbit on the $2\ell$ blocks and
$(ix)^{2\ell}=-1$.  Since $(ix)^\ell=\pm i$, it generates both $i$ and
$x$ over $\Z$, giving the negacyclic description.
\end{proof}

The frame glue carries the corresponding finite weighted coset metric, so
the lattice minimum is a finite weighted code distance in
$\mathcal C_\ell$.  For the proofs below, however, the completed-square
coordinates in \eqref{eq:square} are more convenient.

For a Gaussian vector $z$, let $w_q(z)$ be the sum of the least squares of
its real and imaginary coefficient residues modulo $q$, and put
\[
 W_q(v)=w_q(bv)+6w_q(v).
\]
Choosing least coefficient lifts in \eqref{eq:square} gives
\begin{equation}\label{eq:minformula}
 \min L_\ell=\min\left\{8,\frac18
 \min_{0\ne v\in R/8R}W_8(v)\right\}.
\end{equation}
Since $L_\ell$ is even,
\begin{equation}\label{eq:div16}
 16\mid W_8(v).
\end{equation}
For $\ell\ge3$, $Q(0,x^j)=\ct(c)=\ell+1$.
\section{Elementary coefficient bounds}\label{sec:metric}
Henceforth $\ell\in\{5,7\}$.  For a Gaussian coefficient modulo $2$, its
\emph{binary weight} is the Hamming weight of its real and imaginary
parity bits; for a coefficient vector, binary weight is the sum over
positions.  We first eliminate odd binary weight and even coefficient
vectors before doing the carry calculation.

\begin{lemma}\label{lem:easy}
If $v\bmod2$ has odd binary weight, then $W_8(v)\ge8\ell$. If
$v\ne0\pmod8$ is even, then $W_8(v)\ge64$.
\end{lemma}
\begin{proof}
From \eqref{eq:kbc},
\begin{equation}\label{eq:bexpanded}
 b=(1+3i)+2(1-i)J+4ix-4x^{-1}.
\end{equation}
Modulo four,
$b v=(1-i)v+2(1-i)\sum_jv_j$. For odd binary weight the second term has
both real components equal to two modulo four. On one Gaussian coordinate
the combined output and sixfold input costs for parities $00,10,01,11$
are at least $8,8,8,16$. Summing proves the first assertion.

If $v=2z$, then $W_8(v)=4W_4(z)$ and $4\mid W_4(z)$. Odd binary weight in
$z$ gives the preceding bound. For positive even binary weight the input
contributes at least twelve, and the output is nonzero modulo four: by
\eqref{eq:detidentity}, $\bar b b=2$ modulo four, so $bz=0$ modulo four
would imply $2z=0$, contrary to an odd coefficient. Thus $W_4(z)\ge16$.
If $z$ is even and nonzero modulo four, its input contribution is already
at least twenty-four.
\end{proof}

\section{The residue calculation at two}\label{sec:carries}
We now work in
\[
 \Z[i]/2\Z[i]=\F_2[\iota]/(\iota^2-1),\qquad
 \eps=1+\iota,
\]
where $\eps^2=0$.  In the binary weight just defined, the four symbols
$0,1,\iota,\eps$ have weights $0,1,1,2$; write this weight as $\wt$.  Put
$\tau(a+\iota b)=a+b$. Multiplication by $\iota$ exchanges the two bits.

Suppose that every real coefficient of $v$ is $0,1$, or $-1$ and that
$\wt(v\bmod2)$ is even. Write $r=v\bmod2$. Then
\[
 \sigma=\frac{(1-i)\sum_jv_j}{2}\pmod2
\]
is defined. From \eqref{eq:bexpanded},
\begin{equation}\label{eq:carry}
 bv\equiv(1-i)v+4p\pmod8,
 \qquad
 p_j=\sigma+\iota r_j+\iota r_{j+1}+r_{j-1}.
\end{equation}
Moreover
\begin{equation}\label{eq:global}
 \sum_jr_j=\tau(\sigma)\eps.
\end{equation}
We call $(r,p)$ \emph{admissible} if, for some
$\sigma\in\Z[i]/2\Z[i]$, the displayed coordinate formula for every $p_j$
and the global relation \eqref{eq:global} both hold.  No choice of signs
for a lift of $r$ is required.

The local contribution to $W_8(v)/8$ is bounded below by
\begin{equation}\label{eq:costtable}
\begin{array}{c|rrrr}
 e(r,p)&0&1&\iota&\eps\\\hline
 0&0&2&2&4\\
 1&1&2&2&3\\
 \iota&1&2&2&3\\
 \eps&2&2&2&4
\end{array}
\end{equation}
(rows indexed by $r$, columns by $p$). Let
$E(r,p)=\sum_je(r_j,p_j)$. Then
\begin{equation}\label{eq:Ebound}
 E(r,p)\le W_8(v)/8.
\end{equation}
For $r=0$ the output is $4p$; for a one-bit symbol the two components
of $(1-i)v$ have absolute value one; for $r=\eps$, their absolute values
are two and zero. These observations give the table.

To invert the carry equation, put $y=x^{-1}$ and use the coefficient basis
$1,y,\ldots,y^{\ell-1}$ in
$(\Z[i]/2\Z[i])[y]/(y^\ell-1)$.  Define
\begin{equation}\label{eq:finverse}
 f_5=(\eps,1,\eps,0,\iota),\qquad
 f_7=(\eps,\eps,1,\eps,0,\iota,\eps).
\end{equation}

\begin{lemma}[Carry inversion]\label{lem:inverse}
Every admissible pair satisfies
\begin{equation}\label{eq:inverse}
 r=f_\ell p.
\end{equation}
For $\ell=7$,
\begin{equation}\label{eq:boundary}
 \tau(r)=y^2(1+y^3)\tau(p).
\end{equation}
\end{lemma}
\begin{proof}
Put $\lambda=y+\iota+\iota y^{-1}$.  For $\ell=5,7$ one checks directly
that
\[
 f_\ell\lambda=1+\iota J,\qquad f_\ell J=\eps J.
\]
Since $p=\sigma J+\lambda r$ and
$\eps\sigma=\tau(\sigma)\eps$, the global relation \eqref{eq:global}
gives $f_\ell p=r$. Applying $\tau$ for $\ell=7$ gives
\eqref{eq:boundary}.
\end{proof}

Equation \eqref{eq:boundary} is the boundary map on the $7$-cycle whose
edges join $j$ to $j+3$. A nonempty proper set has boundary of size two if
and only if it is an interval in this cycle.

We need three consequences of \eqref{eq:inverse}.

\begin{lemma}\label{lem:sparse}
\begin{enumerate}
\item[(i)] If $\wt(r)=2$, then $E\ge6$ for $\ell=5$ and $E\ge8$ for $\ell=7$.
\item[(ii)] A one-bit carry has $\wt(r)=2\ell-4$.
\item[(iii)] For $\ell=7$, a two-bit carry at one position $a$ gives
$r=\eps y^a(y^2+y^5)$. If the two bits occur at distinct positions, then
$\wt(r)\ge6$, and $r$ cannot equal $\eps$ at both carry positions.
\end{enumerate}
\end{lemma}
\begin{proof}
For (i), a single symbol $\eps$ gives $E=2\ell$. Otherwise rotate and
multiply by $\iota$ so that $r_0=1$ and
$r_t\in\{1,\iota\}$, $1\le t\le(\ell-1)/2$. Adding the contributions on the
two three-point neighborhoods in \eqref{eq:carry} gives the following
table.  Each brace lists the two possible totals, corresponding to the
two choices of $\sigma$ with the required value of $\tau(\sigma)$:
\[
\begin{array}{c|cc}
 &r_t=1&r_t=\iota\\\hline
 t=1&\{8,4\ell-8\}&\{2\ell-4,2\ell+4\}\\
 t=2&\{12,4\ell-12\}&\{2\ell-4,2\ell+4\}\\
 t\ge3&\{12,4\ell-12\}&\{2\ell\}.
\end{array}
\]
This proves (i).

A one-bit carry gives a translate of $f_\ell$ or $\iota f_\ell$, proving
(ii). For $\ell=7$, $\eps f_7=\eps(y^2+y^5)$, proving the first assertion
of (iii). For carries at distinct positions normalize them to $0,t$.
The only zero of $f_7$ is at $4$, and its non-$\eps$ positions are
$\{2,4,5\}$. Hence at least one of the two reconstructed symbols at
positions $4$ and $4+t$ is $\eps$. The boundary \eqref{eq:boundary} has
weight four except for $t=\pm3$, when it has weight two. In the latter
case both positions $4$ and $4+t$ reconstruct to $\eps$; hence again
$\wt(r)\ge6$. Finally, $r_0=r_t=\eps$ would force
$(f_7)_t=(f_7)_{-t}=0$, impossible because $f_7$ has only one zero.
\end{proof}

\begin{proposition}[Cycle distance]\label{prop:cycle}
For every nonzero admissible pair $(r,p)$,
\[
 E(r,p)\ge\ell+1,
 \qquad \ell=5,7.
\]
\end{proposition}
\begin{proof}
Put $w=\wt(r)$, $\nu=\wt(p)$, and
$d=|\{j:r_j=\eps\}|$. From \eqref{eq:costtable},
\begin{equation}\label{eq:budget}
 E\ge w,\qquad E\ge w+\nu-d.
\end{equation}
Both $w$ and $E$ are even. For the latter, modulo two the cost is
\[
 \sum_j\tau(r_j)\bigl(1+\tau(p_j)\bigr).
\]
Substitution of \eqref{eq:carry}, together with
$\sum_j\tau(r_j)=0$, cancels the two neighbor sums and gives zero. If $E=w$, the carry is supported on the
$d$ positions where $r_j=\eps$, and every nonzero carry symbol has one bit.
Weight zero gives either the zero pair or $E=4\ell$, and weight two is
Lemma~\ref{lem:sparse}(i).

For $\ell=5$, only $w=4$ remains under the assumption $E\le4$, hence
$E=w$. If $d=0$, \eqref{eq:inverse} gives $r=0$. If $d=1$, a nonzero carry
has one bit and reconstructs weight six. If $d=2$, all input symbols are
$0$ or $\eps$, whence all carry symbols are $0$ or $\eps$, contrary to the
one-bit equality condition.

Let $\ell=7$ and assume $E\le6$. For $w=6$, again $E=w$, so $p$ is supported
on at most $d\le3$ double-symbol positions. The cases $d\le2$ are excluded
by Lemma~\ref{lem:sparse}(ii),(iii). If $d=3$, all input symbols are
$0$ or $\eps$ and \eqref{eq:global} has $\tau(\sigma)=1$; then every carry symbol
has one bit, not merely those three.

It remains to take $w=4$. If $d=0$, \eqref{eq:budget} gives $\nu\le2$,
excluded by Lemma~\ref{lem:sparse}. If $d=2$, all input and carry symbols
are $0$ or $\eps$. The excess budget $E-w\le2$ allows at most one nonzero
carry, at an input position with symbol $\eps$. But an $\eps$-carry
reconstructs zero at its own position, by Lemma~\ref{lem:sparse}(iii).
Thus $d=1$.  Lemma~\ref{lem:sparse} excludes $\nu\le2$; hence
\eqref{eq:budget} forces $\nu=3$, and equality holds.  Let $j_0$ be the
unique position with $r_{j_0}=\eps$.
The carry is zero at every empty input position and has one bit at $j_0$.

If the carry consists of an $\eps$ at $a$ and a one-bit symbol at $j_0$, then
\eqref{eq:boundary} places the two one-bit input symbols at $j_0\pm2$, so
$a=j_0\pm2$. The two carry columns then both contribute $\eps$ at $j_0$, giving
$r_{j_0}=0$, a contradiction.

Otherwise the carry has one-bit symbols at three distinct positions. Their
set $T$ is the entire input support, while $\tau(r)$ has weight two. Hence
$T$ is an interval of length three in the step-three cycle. After translation,
\[
 T=\{0,3,6\},\qquad
 \tau(r)=y^2(1+y^3)(1+y^3+y^6)=y^2+y^4.
\]
Indeed, the step-three cyclic order is
$0,3,6,2,5,1,4,0$: the interval $T=\{0,3,6\}$ has boundary vertices
$2$ and $4$, both outside $T$.  This contradicts the fact that $T$ is the
entire input support.
\end{proof}

\section{The extremal minima}\label{sec:minima}
The remaining estimate concerns one doubled coefficient in dimension $28$.

\begin{lemma}\label{lem:doubled}
For $\ell=7$, suppose that $v$ has exactly two real coefficients $\pm1$,
one real coefficient $\pm2$, and all other coefficients zero. Then
$W_8(v)\ge64$.
\end{lemma}
\begin{proof}
Consider first the sum of the two signed real unit columns of the real
matrix representing multiplication by $b$. Outside their input positions it
has a zero real component in at least two distinct Gaussian positions. For distinct Gaussian input positions, the following
table lists the phase-ratio multiset on the five common off-diagonal
positions. The third column is the largest possible number of aligned
contributions after the two column phases are fixed.
\[
\begin{array}{c|c|c|c}
\text{cyclic separation}&\text{phase ratios}&
 \text{maximum aligned}&\text{guaranteed nonaligned}\\\hline
1&\{1,1,1,i,i\}&3&2\\
2&\{1,1,i,i,-1\}&2&3\\
3&\{1,i,i,-i,-i\}&2&3
\end{array}
\]
(The other separations are obtained by conjugation.) At every nonaligned
position the two contributions sum to zero or to $4$ times a Gaussian
unit; in either case one of the two real components vanishes. The table
therefore supplies at least two distinct zero real components. If the two
real unit columns occupy the same Gaussian position, the same conclusion
holds at every other position.

The column belonging to the doubled coefficient has odd entries only at its
own Gaussian position. Hence at least one of the two zero components above
receives $\pm4$ from the doubled column. Thus $w_8(bv)\ge16$. The input
contributes $6(1+1+4)=36$, so $W_8(v)\ge52$; by \eqref{eq:div16},
$W_8(v)\ge64$.
\end{proof}

\begin{proposition}\label{prop:minima}
The lattice $L_5$ has minimum $6$, and $L_7$ has minimum $8$.
\end{proposition}
\begin{proof}
A counterexample would have $W_8(v)\le32$ for $\ell=5$ or $W_8(v)\le48$
for $\ell=7$, by \eqref{eq:minformula} and \eqref{eq:div16}. Hence its least real coefficient lift
has squared weight at most $5$ or $8$, respectively; in particular no
coefficient $\pm3$ occurs. Lemma~\ref{lem:easy} reduces to positive even
binary weight.

If there is no doubled coefficient, Proposition~\ref{prop:cycle} and
\eqref{eq:Ebound} give $W_8(v)\ge48$ or $64$. For $\ell=5$, one doubled
coefficient together with the two required odd coefficients already has
input squared weight $6>5$.

Let $\ell=7$. Two doubled coefficients and two odd coefficients have input
weight $10>8$. With one doubled coefficient there are two or four odd
coefficients. The first case is Lemma~\ref{lem:doubled}. Four odd
coefficients contribute $48$ from the input. The output is nonzero modulo
eight: \eqref{eq:detidentity} gives $\bar b b=2$ modulo eight, so
$bv=0$ modulo eight would imply $2v=0$, impossible in the presence of an
odd coefficient. Hence $W_8(v)>48$.

Thus the lower bounds are proved. Since $Q(0,x^j)=\ell+1$, equality holds.
\end{proof}

\begin{proof}[Proof of Theorem~\ref{thm:main}]
Proposition~\ref{prop:construction} gives evenness, strong
$6$-modularity, determinant, and type $\typea$ for $L_5$ and $L_7$.
Proposition~\ref{prop:glue} gives the transitive frames, and
Proposition~\ref{prop:minima} gives the minima. These equal the bound
\eqref{eq:bound}.
\end{proof}

\section{\texorpdfstring{The rank-$12$ member and the Coxeter--Todd lattice}{The rank-12 member and the Coxeter--Todd lattice}}\label{sec:k12bridge}
The rank-$12$ member can be compared directly with the classical hexacode
construction.  The Coxeter--Todd lattice $K_{12}$ is the Construction-A
lift of the Hermitian self-dual $[6,3,4]_4$ hexacode from an
$A_2(2)^6$-frame \cite{ConwaySloane99,SelfDual}.  The code has order
$4^3=64$, so
\[
 \frac{\det A_2(2)^6}{64^2}=\frac{12^6}{64^2}=3^6=\det K_{12},
\]
and its minimum distance four excludes roots in the lift.  The next
proposition identifies an index-$8$ overlattice of $L_3$ with $K_{12}$.

\begin{proposition}\label{prop:L3K12}
The lattice $L_3$ has minimum $4$.  Put
\[
 h=\frac{1+i}{2}(1,x)\in L_3^\#.
\]
The lattice
\begin{equation}\label{eq:K12over}
 \widetilde L_3=L_3+\Z h+\Z xh+\Z x^2h
\end{equation}
has index $8$ over $L_3$, determinant $3^6$, and minimum $4$.
Consequently $\widetilde L_3$ is isometric to the Coxeter--Todd lattice
$K_{12}$.
\end{proposition}
\begin{proof}
For $\ell=3$ one has
\[
 b=3+i+(2+2i)x-(2+2i)x^2.
\]
A vector of norm two in $L_3$ would require $W_8(v)=16$, hence
$w_8(v)\le2$.  Up to a Gaussian unit and a cyclic shift its least lift is
one of
\[
\begin{array}{c|rrrrrr}
 v&1&1+i&1+x&1-x&1+ix&1-ix\\\hline
 W_8(v)&32&64&32&64&48&48.
\end{array}
\]
Thus $\min L_3\ge4$, while $Q(0,1)=4$.

A direct calculation from \eqref{eq:H} gives
$(x^rh,L_3)\subseteq\Z$ for $r=0,1,2$ and
\begin{equation}\label{eq:hgram}
 \bigl((x^rh,x^sh)\bigr)_{r,s=0}^2=
 \begin{pmatrix}4&2&2\\2&4&2\\2&2&4\end{pmatrix}.
\end{equation}
The three classes are independent modulo $L_3$.  Indeed, if
$\alpha_0h+\alpha_1xh+\alpha_2x^2h\in L_3$ with
$\alpha_r\in\{0,1\}$, then its first coordinate is
\[
 \frac{1+i}{2}\bigl(\alpha_0+\alpha_1x+\alpha_2x^2\bigr)\in R,
\]
which is possible only when all three $\alpha_r$ vanish.  Hence
$[\widetilde L_3:L_3]=8$.  Equation~\eqref{eq:hgram} also shows that
$\widetilde L_3$ is even, and therefore
\[
 \det\widetilde L_3=\frac{6^6}{8^2}=3^6.
\]

It remains to exclude norm two in the seven new cosets.  Up to cyclic
shift these cosets have representatives
\[
 h,\qquad h+xh,\qquad h+xh+x^2h.
\]
Write a vector in one of these cosets as $(u,v)$ and put
$U=2u$, $V=2v$.  Then
\begin{equation}\label{eq:K12square}
 32Q(u,v)=\norm{8U+bV}_0^2+6\norm V_0^2.
\end{equation}
Because $\widetilde L_3$ is even, a nonzero vector of norm below $4$
would have norm exactly $2$.  Equation~\eqref{eq:K12square} would then
have right-hand side $64$, so necessarily $\norm V_0^2\le10$.

Use the cyclic coefficient order $1,x,x^2$.  In a coefficient called
``odd--odd'' both its real and imaginary parts are odd; all unlisted
coefficients are even--even.  Under the bound $\norm V_0^2\le10$, an
even--even Gaussian coefficient belongs to
\[
 \{0,\ \pm2,\ \pm2i,\ \pm2\pm2i\},
\]
with squared norms $0,4,8$, while an odd--odd coefficient belongs to
\[
 \{\pm1\pm i,\ \pm1\pm3i,\ \pm3\pm i\},
\]
with squared norms $2,10$.  Thus the local counting polynomials are
\[
 E(t)=1+4t^4+4t^8,
 \qquad
 O(t)=4t^2+8t^{10}.
\]
The three cosets impose the parity patterns $EOE$, $EOO$, and $OOO$ on
$V$, respectively.  Truncating the corresponding products at total
squared norm $10$ gives the counts in the table below.

For the $U$-minimization, if a real component of $U$ is prescribed to
have parity $p\in\{0,1\}$, set
\[
 d_p(a)=\min_{n\equiv p\ (2)}(a+8n)^2.
\]
If $r\in\{0,1,\ldots,15\}$ is congruent to $a+8p$ modulo $16$, then
\begin{equation}\label{eq:dp}
 d_p(a)=\min\{r^2,(16-r)^2\}.
\end{equation}
Hence the $U$-term in \eqref{eq:K12square} is minimized independently in
its six real components by \eqref{eq:dp}.  Direct substitution of the
finite coefficient sets above gives the complete distribution
\[
\begin{array}{c|c|r|l}
\text{coset}&\norm V_0^2&\text{number of }V&
 \text{values of the RHS of \eqref{eq:K12square}}\\\hline
h+L_3&2&4&4\cdot128\\
&6&32&8\cdot128+24\cdot192\\
&10&104&40\cdot128+36\cdot192+28\cdot256\\\hline
h+xh+L_3&4&16&12\cdot128+4\cdot320\\
&8&64&24\cdot128+32\cdot192+8\cdot256\\\hline
h+xh+x^2h+L_3&6&64&12\cdot128+52\cdot192
\end{array}
\]
The six rows exhaust the $140$, $80$, and $64$ possible $V$ in the three
cosets.  In every row the right-hand side is at least $128$, whereas a
norm-$2$ vector would require it to equal $64$.  Thus no new coset
contains a norm-$2$ vector, and $\min\widetilde L_3=4$.

The $2$-power extension does not change the ternary local type, so
$\widetilde L_3$ lies in the genus $II_{12}(3^{+6})$.  J\"urgens records
that the Coxeter--Todd lattice is the unique member of this genus with
minimum four \cite[Example~1.3.1(a), p.~21]{Juergens15}.  Hence
$\widetilde L_3\cong K_{12}$.
\end{proof}

The $A_2(4)^6$-frame of $L_3$ is not the standard $A_2(2)^6$-frame used in
the hexacode construction.  The two code descriptions therefore use
different frames.  Proposition~\ref{prop:L3K12} shows that $L_3$ is a
$2$-primary sublattice of $K_{12}$, while $K_{12}$ retains its classical
$A_2(2)^6$/hexacode presentation.

At the exceptional endpoint $\ell=1$,
$k=-2$, $b=-1+5i$, and $c=4$.  The vector $(-i,1)$ has norm two, so
$L_1$ is the known extremal $4$-dimensional member of the family.

\section*{Acknowledgments}
The author would like to thank Joseph Iosue for, many years ago, introducing
him to the connection between lattices, codes, and conformal field theory,
and Ansgar Burchards for, more recently, introducing him to the connection
between quantum codes, weight enumerators, lattices, and lattice theta
functions.  These conversations sparked the interest that led to this
project.  The author would also like to thank Carl Miller for introducing
him to learning with errors and the idea of lattice cryptography many years
ago, and Daniele Micciancio for a much more recent conversation about the
hardness of finding lattice minima that encouraged the author to learn more
about methods for proving lattice minima.

\section*{Declaration of generative AI and AI-assisted technologies}
The author used OpenAI ChatGPT in the preparation of this article, including
to carry out and check computations (sometimes together with Magma V2.29).  The author takes full responsibility for all
mathematical statements and proofs in this article.

\end{document}